\documentclass{amsart}
\usepackage{amssymb, amscd, amsmath}

\title[Multipliers of cycles of cubic polynomials]{Cubic polynomials with periodic cycles of a specified multiplier}
\author{Patrick Ingram}
\address{Department of Pure Mathematics, University of Waterloo \emph{Current address:} Department of Mathematics, Colorado State University}
\email{pingram@math.uwaterloo.ca \emph{Current e-mail:} pingram@math.colostate.edu}
\thanks{This research was supported in part by a Discovery Grant from NSERC of Canada.}
\date{\today}

\newcommand{\QQ}{\mathbb{Q}}
\newcommand{\ZZ}{\mathbb{Z}}
\newcommand{\CC}{\mathbb{C}}

\newcommand{\PP}{\mathbb{P}}
\renewcommand{\AA}{\mathbb{A}}

\newcommand{\Ocal}{\mathcal{O}}
\newcommand{\Ecal}{\mathcal{E}}

\newcommand{\Pcal}{\mathcal{P}}
\newcommand{\Gal}{\operatorname{Gal}}
\newcommand{\Spec}{\operatorname{Spec}}

\newcommand{\PGL}{\operatorname{PGL}}
\newcommand{\PSL}{\operatorname{PSL}}

\newcommand{\I}{\textup{I}}
\newcommand{\III}{\textup{III}}

\newcommand{\NS}{\operatorname{NS}}
\newcommand{\rank}{\operatorname{rank}}

\newcommand{\XEll}{X^{\mathrm{Ell}}}

\newtheorem{theorem}{Theorem}

\newtheorem{proposition}[theorem]{Proposition}
\newtheorem{lemma}[theorem]{Lemma}

\theoremstyle{remark}
\newtheorem{remark}{Remark}
\newtheorem{question}{Question}

\begin{document}
\begin{abstract}
We consider cubic polynomials $f(z)=z^3+az+b$ defined over $\CC(\lambda)$, with a marked point of period $N$ and multiplier $\lambda$.  In the case $N=1$, there are infinitely many such objects, and in the case $N\geq 3$, only finitely many (subject to a mild assumption).  The case $N=2$ has particularly rich structure, and we are able to describe all such cubic polynomials defined over the field $\bigcup_{n\geq 1}\CC(\lambda^{1/n})$.
\end{abstract}

\maketitle

\section{Introduction}

Let $\widehat{\CC}$ denote the Riemann sphere, and let $f:\widehat{\CC}\rightarrow\widehat{\CC}$ be a holomorphic function.  If one is interested in studying the dynamics of $f$, one natural starting point is to describe the periodic points under $f$.
  The point $\alpha\in\widehat{\CC}$ is said to be a point of \emph{period dividing} $N$ for $f$ if \[f^N(\alpha)=f\circ f\circ\cdots\circ f(\alpha)=\alpha,\]
and a point of \emph{(exact) period} $N$ if $N$ is the least positive integer for which the above relation holds.    If we expand $f^N(z)-\alpha$ as a power series near $z=\alpha$,
\[f^N(z)-\alpha=\lambda(z-\alpha)+c_2(z-\alpha)^2+\cdots,\]
then the coefficient $\lambda$, called the \emph{multiplier} of the periodic point $\alpha$, determines much of the dynamics near the cycle.
  We say that the cycle is
\emph{repelling} if $|\lambda|>1$,
\emph{attracting} if $|\lambda|<1$,
and 
\emph{indifferent} if  $|\lambda|=1$.
This distinction turns out to be fairly important in the classification of the dynamics of holomorphic functions; for example, a classical result in holomorphic dynamics (see \cite[Theorem~14.1]{milnor}) states that the Julia set of a function is exactly the closure its set of repelling periodic points.

The aim of this paper is to make a few observations about periodic points of cubic polynomials.  Cubic polynomials with marked points of period $N$ are parametrized by a 2-dimensional moduli space $\mathcal{P}_{3}(N)$ (defined more precisely below).  The map $\lambda:\mathcal{P}_3(N)\rightarrow \widehat{\CC}$ taking a cycle of a polynomial to its multiplier offers a natural fibration of these spaces.  The fibres of the multiplier map turn out to be of intrinsic interest, arising in the the classification of the hyperbolic components of the connectedness locus of the moduli space of cubic maps (see, for example, \cite{milnor_cubics}).

More specifically, we are interested in describing sections, and certain multi-sections, of these fibrations.
By a \emph{multiplier section of period $N$} we mean a triple of holomorphic functions $a, b, z_1:\widehat{\CC}\rightarrow \widehat{\CC}$ in the variable $\lambda$, such that $z_1$ is a point of period $N$ 
for the map $f(z)= z^3+az+b$, with multiplier $\lambda$.
For example, the cubic map
\[f(z)=z^3+\lambda z\]
has a fixed point $z_1=0$, with multiplier $\lambda$.
  More generally, for a morphism $\lambda:X\rightarrow \widehat{\CC}$ of compact Riemann surfaces, one might ask about meromorphic triples $a, b, z_1:X\rightarrow\widehat{\CC}$ of the same sort, again with multiplier $\lambda$.  One particularly natural case of this is triples $a, b, z_1:\widehat{\CC}\rightarrow\widehat{\CC}$ in the variable $w$, defining an $N$-cycle of multiplier $\lambda=w^m$, since the distinction between attracting, repelling, and indifferent cycles is defined identically in terms of $|w|$ as in terms of $|\lambda|$. Triples of this form will be called
\emph{$m$th-root multiplier sections of period $N$}, and these will be called \emph{primitive} unless they factor non-trivially through an intermediate map $\lambda=u^d$.
An example of a square-root multiplier section of period $2$ is the given by the cubic polynomial
\[f(z)=z^3+\frac{1}{6}(w^2-9)z+\frac{\sqrt{-2}}{54}(w^2-9)w,\]
which has a point of period 2 at
\[z_1=\frac{-\sqrt{-2}}{6}(w+3i)\]
with multiplier $\lambda=w^2$. 

Our first theorem is fairly elementary, but is presented for contrast with Theorems~\ref{th:N=2} and \ref{th:N>3}.

\begin{theorem}
There exist infinitely many multiplier sections of period 1.
\end{theorem}

We can, in fact, simply write down an explicit parametrization of all such sections.
The case $N=2$ turns out to be much more interesting.  Not only are there no multiplier sections in this case, but the full set of $m$th-root multiplier sections can be described fairly explicitly.

\begin{theorem}\label{th:N=2}
If $(a, b, z_1)$ is a primitive $m$th-root section of period 2, then $m$ divides 12.  Furthermore, the set of triples of this form is infinite, but has the structure of a finitely generated abelian group of rank 3.  Finally, there are no multiplier sections  (that is, $m$th root sections with $m=1$) of period 2.
\end{theorem}

While there are infinitely many $12$th-root multiplier sections of period 2, it turns out that another sort of finiteness holds (beyond the finite generation of the group of sections): given any finite set of points $S\subseteq\widehat{\CC}$, only finitely many of these sections have no poles outside of $S$.

In the case where $N\geq 3$, it turns out that there are only finitely many multiplier sections in total (for each $N$), assuming the above-mentioned fibration is not isotrivial.  Recall that a fibred suface is \emph{isotrivial} if all smooth fibres are isomorphic.  The fibration in question is non-isotrivial in the cases  $N=1, 2$ and $3$, and we suspect this to be true for all $N$.

\begin{theorem}\label{th:N>3}
Let $N\geq 3$, and suppose that the fibration of $\mathcal{P}_3(N)$ by $\lambda$ is not isotrivial.  Then there exist only finitely many multiplier sections of period $N$.  More generally, given any compact Riemann surface $X$ and  holomorphic map $\lambda:X\rightarrow\widehat{\CC}$, there exist only finitely many triples $a, b, z_1:X\rightarrow \widehat{\CC}$ as above with period $N$ such that the resulting cycle has multiplier $\lambda$.
\end{theorem}

Both of these results lead us to ask the following question:
\begin{question}
Do there exist \emph{any} multiplier sections of period $N\geq 3$?  What about $m$th-root multisections, where $m$ is arbitrary?
\end{question}

As we are interested in sections of the fibred surfaces $\lambda:\mathcal{P}_3(N)\to\widehat{\CC}$, it is reasonable to consider the generic fibres, on which these sections correspond to points.  Specifically, the function field of $\mathcal{P}_3(N)$ has transcendance rank 1 over the field $\CC(\lambda)$, and hence is the function field of some smooth, projective, algebraic curve $X_1(N)$ over this field.  Sections of the fibration correspond to $\CC(\lambda)$-rational points on $X_1(N)$, and so tools from arithmetic geometry may be brought to bear; it is this connection which we use to prove the theorems above.  It is worth noting that our results also show, for example, that if $N\geq 3$, and $K$ is a number field, then for all but finitely many $\lambda\in K$ (that is, all but those below ``bad fibres'' of the moduli space), there exist at most finitely many cubic polynomials $z^3+az+b$ with $a, b\in K$ and a marked $K$-rational point of period $N$.  
Since points on $X_1(N)$ correspond to cubic polynomials with a marked point of period $N$, there is a natural action of $\ZZ/N\ZZ$ on this curve which sends $(f, P)$ to $(f, f(P))$.  The quotient by this group of automorphisms will be denoted by $X_0(N)$, a notation intended to be evocative of the analogous moduli problem in the study of elliptic curves.

Note that, while the fibrations above admit few sections, the underlying surfaces in some cases are fairly simple.  For example, in the case $N=2$, the multiplier fibration is elliptic (that is, the generic fibre $X_1(2)$ is an elliptic curve over $\CC(\lambda)$), but the underlying space is rational.  This means, for example, that if $K$ is a number field, then cubic polynomials $z^3+az+b$ with $K$-rational coefficients,  and a $K$-rational point of period 2, are relatively common, while such pairs with a specified multiplier are relatively sparse.  It turns out that the moduli spaces of polynomials with  marked points of relatively small period is always rational.  

\begin{theorem}\label{th:rational}
Fix and integer $d\geq 2$, and natural numbers $N_1, ..., N_s$ with 
\[N_1+N_2+\cdots+N_s\leq d+1.\]   Then the fibre product of the spaces $\mathcal{P}_d(N_1), \ldots, \mathcal{P}_d(N_s)$, over the moduli space of polynomials of degree $d$, that is, the moduli space parametrizing polynomials of degree $d$ with marked points of period $N_1, ..., N_s$, is rational.
\end{theorem}

  This prompts some obvious questions:
\begin{question}
Is it true that $\mathcal{P}_d(N)$ is rational \emph{only if} $N\leq d+1$?  Is it true that there is some $M=M(d)$ such that $N\geq M(d)$ implies that $\mathcal{P}_d(N)$ is a variety of general type, and if so, what is the least such $M$ for each $d$?
\end{question}

The paper is organized as follows.  In Section~1, we define the moduli spaces under discussion formally, and establish some of their basic properties.  Although the language of the introduction is largely that of complex manifolds, we shift notation here into the language of algebraic/arithmetic geometry.  In Section~2, we write down an explicit model of the generic fibre in the $N=1$ case.  In Section~2 we treat the $N=2$ case.  Here, the generic fibre of our moduli space is a curve of genus 1.  In particular, describing the sections amounts to describing points on an elliptic curve over $\CC(\lambda)$.  Proving Theorem~\ref{th:N=2}, however, requires us to describe the group of points on this curve over the infinite procyclic extension $\bigcup_{n=1}^\infty \CC(\lambda^{1/n})$.  \emph{A priori}, the group of points on an elliptic curve over a field like this might not be finitely generated. As it transpires, though, the elliptic curve in question satisfies the conditions of a result of Fastenberg \cite{fastenberg}, and so we are able to completely describe the points on the curve over this field.  The finiteness claim following Theorem~\ref{th:N=2}, then, follows from Siegel's Theorem in function fields (which is made explicit in \cite{silvhindry}).  In Section~3 we show that the generic fibre of the moduli space is a curve of genus at least 2, for $N\geq 3$.  This proves Theorem~\ref{th:N>3}, given that Mordell's Conjecture holds in function fields (see, for example, \cite{vojta}).  In Section~4 we look into moduli spaces of polynomials of higher degree, proving Theorem~\ref{th:rational}.

Although we have chosen to remain relatively specific in this paper, and focus on cubic polynomials, much of what is done could be done for any two-parameter family of polynomials.  For example, the author worked out many analogous results for the family of biquadratic maps $f(z)=(z^2+a)^2+b$.  Similarly, it is possible to discuss the fibered surface of quadratic rational maps with a marked point of period $N$ (see \cite{epstein}, where it is shown that the multiplier fibration of the moduli space of quadratic rational maps with a marked point of period 3 is elliptic).


\section{Moduli spaces}

Our theorems are proved by constructing various curves over $K=\CC(\lambda)$, and examining the  points on these varieties rational either over $K$, or some extension of $K$.  These curves are the generic fibres of various fibred surfaces, but we leave the study of the underlying surfaces, for the most part, to future work.  First, we will discuss moduli spaces in general.

The standard moduli space of polynomials of degree $d$ is constructed as follows.  First, to each $\overline{a}=(a_d, ..., a_0)\in\AA^{d+1}$, we associate the polynomial
\[f_{\overline{a}}(z)=\sum_{0\leq i\leq d} a_iz^i.\]
To make things invariant of the choice of coordinates, we will take the quotient of this by the group of affine transformations $\phi(z)=\alpha z+\beta$.  This group  acts on the polynomials above by conjugation \[f^\phi= \phi\circ f\circ\phi^{-1},\] and the quotient variety will be called $\Pcal_d$.  The question of how to compactify this space is interesting, but beyond to scope of this paper.  Note that one might, equivalently (and probably more naturally), define $\Pcal_d$ to be the quotient of the space  of all \emph{rational} functions on $\PP^1$ with a totally ramified fixed point, modulo the action of conjugation by the full group of automorphisms of $\PP^1$.

  The moduli space $\Pcal_d(N)$, of polynomials with a marked point of period $N$, is defined similarly.  We let $\phi(x)=\alpha x+\beta$ act on $(a_d, ..., a_0, z)\in \AA^{d+2}$ by the action described above, for the first $d+1$ coordinates, and by $z^\phi=\phi(z)$.  Now, we may define polynomials $\Phi_d(a_d, ..., a_0, z)$ by
\[f_{\overline{a}}^N(z)-z=\prod_{k\mid N}\Phi_k(a_d, ..., a_0, z),\]
solutions of which correspond to polynomials with a marked point of formal period $N$ (see, e.g., \cite[p.~149]{ads}).  The quotient of the variety  $\{\Phi_N=0\}$ by the action of the affine transformations is $\Pcal_d(N)$.

Unfortunately, these moduli spaces do not interact particularly well with obvious normal forms for polynomials.  It is not uncommon to normalize polynomials so that they are monic, and the average of their roots (the \emph{barycenter}) vanishes:
\[f(z)=z^d+a_{d-2}z^{d-2}+\cdots a_1z+a_0.\]
Every polynomial of this form has degree $d$, and every polynomial of degree $d$ is affine-conjugate to one of this form.
Unfortunately, the affine transformation $z\mapsto \zeta z$, for $\zeta$ a primitive $(d-1)$th root of unity, acts non-trivially on polynomials of this form, and so the conjugacy class of the polynomial being defined over a given field is not the same as the coefficients being defined over that field (in other words, the field of moduli is often a proper subfield of the field generated by the coefficients in this particular normal form).  One might opt to use the normal form
\[f(z)=a_dz^d+a_{d-2}z^{d-2}+\cdots+a_1z+1,\]
where the field of moduli truly is the field generated by the coefficients, but this is also unsatisfactory.  This normal form offers only a birational correspondance between the space of coefficients and the space of polynomials.  The tuples of coefficients $(0, a_{d-2}, ... ,a_1, a_0)$ do not correspond to polynomials (of degree $d$), while polynomials which fix their own barycenter are not represented in this form.

To resolve this conflict, we define somewhat less high-brow moduli spaces for cubic polynomials, simply insisting on the normal form $f_{a, b}=z^3+az+b$. We will define a variety $\mathcal{Y}_1(N)$ as follows.  Let $Z\subseteq \AA^{N+2}_\CC$ be the variety defined by
\begin{gather}
f_{a, b}(z_1)-z_2=0\nonumber\\
f_{a, b}(z_2)-z_3=0\label{eq:Ycal}\\
\vdots\nonumber\\
f_{a, b}(z_N)-z_1=0\nonumber,
\end{gather}
which is clearly birational to the subvariety of $\AA^3$ defined by $f_{a, b}^N(z)-z=0$.  If $\Phi_N(a, b, z)$ is defined by
\[f^N_{a, b}(z)-z=\prod_{d\mid N}\Phi_d(a, b, z),\]
as above, then we will let $\mathcal{Y}_1(N)\subseteq Z$ be the variety corresponding, under this birational map, to the component $\Phi_N(a, b, z)=0$.  Now, we will let $Y_1(N)$, the generic fibre, be the $\CC[\lambda]$-scheme obtained by mapping $\CC[\lambda]$ into $\CC[\mathcal{Y}_1(N)]$ by
\[\lambda\mapsto f'_{a, b}(z_1)f'_{a, b}(z_2)\cdots f'_{a, b}(z_n).\]
(where the differentiation is with respect to $z$).  In other words, $Y_1(N)$ corresponds to the appropriate component of the subvariety of $\AA^{N+2}_{\CC[\lambda]}$ defined by the equations~\eqref{eq:Ycal}, along with the additional equation
$f'_{a, b}(z_1)\cdots f'_{a, b}(z_N)-\lambda=0$.

\begin{remark}
  Note that polynomials above all have coefficients in $\ZZ$, and so we could have defined $Y_1(N)$ as a  $\ZZ[\lambda]$ scheme.  While these objects are certainly worth studying, we focus our initial investigations to geometric properties, and so work over $\CC$ for simplicity.
  \end{remark}

Now, let $\hat{f}$ be the automorphism of $Y_1(N)$ defined by \[(a, b, z_1, ..., z_N)\mapsto (a, b, z_2, ..., z_N, z_1).\]  We will let $Y_0(N)$ denote the quotient of $Y_1(N)$ by this automorphism, and we will let $X_1(N)$ and $X_0(N)$, respectively, be smooth projective curves birational to $Y_1(N)$ and $Y_0(N)$.

  The curve $Y_0(N)$ parametrizes cubic polynomials $f(z)=z^3+az+b$ with marked cycles (rather than points) of period $N$.  In particular, recalling that $K=\CC(\lambda)$, $K$-rational points on $Y_0(N)$ correspond to cubic polynomials in $K[z]$ with marked cycles of period $N$, fixed setwise (but not necessarily pointwise) by the absolute Galois group $\Gal(\overline{K}/K)$.  

We will also define two curves $P_1(N)$ and $P_0(N)$, which will be the quotient of $X_1(N)$ and $X_0(N)$ by the automorphism induced by 
\[(a, b, z_1, ..., z_N)\mapsto (a, -b, -z_1, ..., -z_N).\]
Thus, $P_1(N)$ is precisely the generic fibre of the surface $\Pcal_3(N)$ (the non-na\"{i}ve moduli space), under the multiplier fibration.
Finally, we will make reference to the curves $X'_1(N)$, $X'_0(N)$, $P'_1(N)$, and $P'_0(N)$, which are the corresponding curves for \[f_{-3u^2, 2v^3}(z)=z^3-3u^2z+2v^3.\]

The following lemma tells us that the variety $Y_1(N)$ is always smooth.  More generally, it says that the variety parametrizing fixed points of any generic polynomial, with transcendental multiplier, is smooth.  A similar argument shows that the variety defined by $\Phi_3(a, b, z)=0$ and $(f_{a, b}^N)'(z)-\lambda=0$ is also non-singular, and so the birational map of affine varieties mentioned above is actually an isomorphism.
The actual statement of the lemma is slightly more general, since we will need this form later.

\begin{lemma}\label{lem:nonsingular}
Let $R$ be a Dedekind domain, let  $P\in R[a_1, ..., a_s, z]$ be a polynomial, let $\mu, \nu\in R$ be non-zero, let $t$ be transcendental over $R[a_1, ..., a_s, z]$, and let 
 $V\subseteq\AA^{s+n}_{R[t]}$ be the variety defined  by the equations 
 \begin{gather*}
 P(a_1, ..., a_s, z_1)-z_2=0\\
 P(a_1, ..., a_s, z_2)-z_3=0\\
 \vdots\\
 P(a_1, ..., a_s, z_n)-\mu z_1=0\\
 \prod_{i=1}^n \frac{\partial P}{\partial z}(a_1, ..., a_s, z_i)-\nu t=0.
 \end{gather*}
Then $V$ is non-singular.
\end{lemma}

\begin{proof}
To simplify notation, let $G_i$ denote the polynomial $P(a_1, ..., a_s, z_i)-z_{i+1}$ for $i\leq n-1$, and $G_n$ denote $P(a_1, ..., a_s, z_n)-\mu z_1$.  We will also let $\Lambda$ stand for the product $\prod_{i=1}^n \partial P/\partial z(a_1, ..., a_s, z_i)$ (as a function on $\AA^{s+n}$).  We will refer to $z_1, ..., z_n$ as $a_{s+1}, ..., a_{s+n}$ wherever it simplifies indexing.

Suppose that $V$ is singular, and let $Q\in V(\overline{\CC(t)})$ be a singular point.  By definition, we have $G_i(Q)=0$ for all $i$, and $\Lambda(Q)=\nu t$.  On the other hand, since $Q$ is a singular point, the Jacobian matrix of $V$ must have rank less than $n+1$ at $Q$.  Therefore, we must have some $\beta_1, ..., \beta_{n+1}\in \overline{\CC(t)}$, not all 0, such that
\[\sum_{i=1}^n\beta_i\frac{\partial G_i}{\partial a_j}(Q)+\beta_{n+1}\frac{\partial\Lambda}{\partial a_j}(Q)=0\]
for each $j$.  Note, the fact that we may consider $\partial \Lambda/\partial a_j$, in the above, follows from the observation that $\partial (\nu t)/\partial a_j=0$ for all $j$.  

First, we will show that $\beta_{n+1}=0$.  To see this, consider the equality 
\begin{equation}\label{eq:the G_i}\sum_{i=1}^n \beta_iG_i(Q)+\beta_{n+1}\Lambda(Q)=\beta_{n+1}\nu t.\end{equation}
This is an equality of functions in $t$, and so we may differentiate with respect to $t$.
Differentiating the right-hand-side of \eqref{eq:the G_i} in terms of $t$, one obtains
\[\beta_{n+1}\nu +\nu t\frac{d\beta_{n+1}}{dt}.\]
On the left-hand-side of \eqref{eq:the G_i}, one obtains
\begin{multline*}
\sum_{i=1}^nG_i(Q)\frac{d\beta_i}{dt}+\Lambda(Q)\frac{d\beta_{n+1}}{dt}  +\sum_{i=1}^n\beta_i\frac{dG_i(Q)}{dt}+\beta_{n+1}\frac{d\Lambda(Q)}{dt}\\
=\nu t \frac{d\beta_{n+1}}{dt}+\left(\sum_{i=0}^n\beta_i\sum_{j=1}^{s+n}\frac{\partial G_i}{\partial a_j}(Q)\frac{d a_j(Q)}{d_t}\right)+\sum_{j=1}^{s+n}\frac{\partial \Lambda}{\partial a_j}(Q)\frac{da_j(Q)}{dt}\\
=\nu t+\sum_{j=1}^{s+n}\frac{da_j(Q)}{dt}\left(\sum_{i=1}^n\beta_i\frac{\partial G_i}{\partial a_j}(Q)+\beta_{n+1}\frac{\partial \Lambda}{\partial a_j}\right)\\
= \nu t \frac{d\beta_{n+1}}{dt},
\end{multline*}
by the definition of the $\beta_i$.  In other words,
\[ \nu t \frac{d\beta_{n+1}}{dt}= \beta_{n+1}\nu +\nu t\frac{d\beta_{n+1}}{dt},\]
as functions of $t$, implying $\beta_{n+1}=0$, given that $\nu\neq 0$.

Thus, we've shown that $\beta_{n+1}=0$, and so $Q$ is in fact a singular point of the variety defined by just the first $n$ equations.  If the Jacobian matrix of this variety has rank less than $n$ at $Q$, though, it certainly implies that the matrix
\[\left(\begin{array}{ccccc}
\frac{\partial P}{\partial z_1}(Q) & -1 & 0& \cdots & 0\\
0 & \frac{\partial P}{\partial z_2}(Q) & -1 & \cdots & 0\\
& \vdots & & & \\
-\mu & 0 & 0 & \cdots & \frac{\partial P}{\partial z_n}(Q)
\end{array}\right)\]
is singular (since this $n\times n$ matrix is a sub-matrix of the Jacobian).  But this matrix has determinant
\[\prod_{i=1}^n\frac{\partial P}{\partial z_i}(Q)-(-1)^n\mu=\Lambda(Q)-(-1)^n\mu.\]
Since $Q$ satisfies $\Lambda(Q)=\nu t\neq (-1)^n\mu$, we have that $Q$ is a non-singular point of $Y$.
\end{proof}

\begin{remark}
Note that the proof above shows that the affine variety parametrizing all fixed points of $P(a_1, ..., a_s, z)$ is singular only on the fibre $\Lambda=1$.  Unfortunately, the projective closure of this variety has many and mysterious singularities at infinity.
\end{remark}

Our next task is to show that the curves $X_1(N)$ and $X_0(N)$ are geometrically irreducible, that is, irreducible over the algebraic closure of $\CC(\lambda)$.  

\begin{proposition}\label{prop:X_1(N) irred}
The curves $X_1(N)$ and $X_0(N)$ are geometrically irreducible.
\end{proposition}

\begin{proof}
Let $S$ be any smooth, projective, irreducible surface over an algebraically closed field, and let $\pi:S\rightarrow C$ be a fibration of $S$.  The generic fibre of the fibration is reducible if and only if the fibration factors as
\[S\stackrel{\pi'}{\longrightarrow} C'\stackrel{\phi}{\longrightarrow} C,\]
for some morphism of curves $\phi:C'\rightarrow C$ of degree greater than one (see, for example, \cite[p.~139]{shaf}). 
In particular, if the surface admits a section $\sigma:C\rightarrow S$, then the generic fibre must be irreducible, since the identity map $\pi\circ\sigma:C\rightarrow C$ cannot factor non-trivially.  Note that $S$ admits a section if and only if the generic fibre has a point rational over $\CC(C)$.

It follows from work of Morton \cite{morton} that $\mathcal{Y}_1(N)$ is irreducible.  
  To show that $X_1(N)$ is irreducible, then, it suffices to show that $X_1(N)(K)$ is non-empty.

The projective variety defined over $\CC(\lambda)$ by
\begin{gather}
z_1^3-3u^2z_1+2v^3-z_2s^2=0\nonumber\\
z_2^3-3u^2z_2+2v^3-z_3s^2=0\label{eq:X' singular}\\
\vdots\nonumber\\
z_{N}^3-3u^2z_{N}+2v^3-z_1s^2=0\nonumber\\
3^N(z_1^2-u^2)(z_2^2-u^2)\cdots(z_{N}^2-u^2)-\lambda s^{2N}=0
\end{gather}
contains a component birational to $X'_1(N)$, and this component 
 has a $\CC(\lambda)$-rational point at 
\[P=[u, v, s, z_1, ..., z_N]=[1, 1, 0, 1, -2, ..., -2].\]
Furthermore, one checks rather easily that the Jacobian matrix of the variety at this point is
\[\left(\begin{array}{ccccc}
0 & 0 & 0 & \cdots & -6\\
0 & 9 & 0 & \cdots & 12 \\
0 & 0 & 9 & \cdots & 12 \\
& \vdots & & \ddots &\vdots \\
6\cdot 9^{N-1} & 0 &  0 & \cdots & -6\cdot 9^{N-1} \\
\end{array}\right),\]
which is non-singular.  Consequently, $P$ corresponds to a $\CC(\lambda)$-rational point on the normalization $X'_1(N)$.  The map induced by $a=-3u^2$, $b=2v^3$ sends this to a $\CC(\lambda)$-rational point on the curve $X_1(N)$.

The irreducibility of $X_0(N)$ simply follows from it being a quotient of $X_1(N)$.
\end{proof}

\section{The case $N=1$}

The space of cubic polynomials with a marked fixed point turns out, unsurprisingly, to be fairly easy to describe.

\begin{proposition}
The curve $X_1(1)=X_0(1)$ is birational, over $\CC(\lambda)$, to $\PP^1$.  The rational parametrization is given by
\begin{gather*}
a = -27s^2+\lambda,\\
b=-54s^3-3s+3\lambda s,\\
z=-3s,
\end{gather*}
for $s\in \PP^1$.
\end{proposition}

\begin{proof}
The curve is described by the two equations
\begin{gather*}
\Phi_1(z, a, b)=f(z)-z=z^3+(a-1)z+b=0\\
\intertext{and}
f'(z)-\lambda=3z^2+a-\lambda=0.
\end{gather*}
Eliminating the variable $z$ (via resultants) we obtain the relation
\[27b^2+(a-t)(2a-3+\lambda)^2=0,\]
a nodal cubic curve over $\CC(\lambda)$.  Setting $u=2a-3+\lambda$, and blowing up at $(b, u)=(0,0)$ by setting $b=sw$, $u=w$, we obtain two components: $w=0$ (with multiplicity 2; this is the exceptional curve), and
\[27s^2-\frac{1}{2}(3\lambda-3-w)=0.\]
This yields
\[a=-27s^2+\lambda, \qquad b=-54s^3-3s+3\lambda s.\]
We may now solve $f(z)-z=0$ for the fixed point:
\[f(z)-z=(z+3s)(z^2-3z s-18s^2-1+\lambda).\]
\end{proof}

Note that the map above gives an isomorphism of the surface $\mathcal{Y}_1(N)$, defined by \[z^3+az+b-z=0,\] with the affine plane $\AA^2$, where the multiplier is sent to one of the two coordinates.  Thus, the smooth projective model of this surface, which is minimal relative to the multiplier fibration, is isomorphic to $\PP^1\times\PP^1$ with projection onto the second coordinate.  Note that this also gives us an explicit description of $\Pcal_3(1)$.  The action of $\PSL_2$ on $\mathcal{Y}_1(1)$ is exactly the map $(s, t)\mapsto (-s, t)$ on $\PP^1\times\PP^1$ as above.  In particular, the map $(s, t)\mapsto (s^2, t)$ gives a map to $\Pcal_3(1)\cong\PP^1\times\PP^1$.


\section{The case $N=2$}

The case $N=2$ is somewhat richer and more interesting than the case $N=1$.  Here, the parametrizing curves $X_1(2)$ and $X_0(2)$ have genus one; they are, in fact, non-isotrivial elliptic curves over $\CC(\lambda)$.  In general, this means that for any compact Riemann surface $X\rightarrow \hat{\CC}$, the set of points on $X_0(2)$ or $X_1(2)$ over $\CC(X)$ has the structure of a finitely generated abelian group, although the structure of this group depends a great deal on the particular covering $X\rightarrow\hat{\CC}$.  It turns out, quite surprisingly, that we can describe this group explicitly for $X=\hat{\CC}\rightarrow\hat{\CC}$ by $\lambda=w^m$, for any $m$.

\begin{proposition}\label{X_1(2) bir}
The curves $X_0(2)$ and $X_1(2)$, respectively, are isomorphic over $\CC(\lambda)$ to the curves
\[E_0:v^2=u(u^2+2u+1-\lambda)\]
and
\[E_1:e^2=d(d^2-4d+4\lambda),\]
and the natural map $X_1(2)\rightarrow X_0(2)$ induces the isogeny $E_1\rightarrow E_0$ with kernel generated by $(0, 0)$.
\end{proposition}

The birational maps to the affine models $Y_1(2)$ and $Y_0(2)$ are given by
\[a=\frac{4u^2-4u+1-\lambda}{6u}\]
\[b=\frac{\sqrt{-2}(8u^2+16u+\lambda-1)v}{54u^2},\]
and
\[z=\frac{\sqrt{-2}(d^2-6d +8\lambda)}{6e}.\]
Note that $z$ is defined only on $E_1$ for obvious reasons, while the maps $a, b:E_1\rightarrow Y_1(2)$ are defined by composition with the isogeny.  Note, as well, that the functions $a, b\in K(E_0)$ have poles precisely at the ``obvious'' points on $E_0$, that is, the point at infinity, and the point $(0, 0)$.  In particular, these points do not lead to cubic polynomials, which would contradict our claim that there are no multiplier sections of period $2$.
It is also worth noting, with a view to analogous problems over function fields, that the above birational maps are defined over $\QQ(\lambda, \sqrt{-2})$.

The remainder of this section will be devoted to uncovering the arithmetic of these curves over the field $K_\infty=\bigcup_{n\geq 1}\CC(\lambda^{1/n})$, which we do largely through the application of a theorem of Fastenberg \cite{fastenberg}, with some minor improvements.  (This appears to be the first time that Fastenberg's result has been used in a ``natural setting''.)  In general, it is not at all clear that the group of points on a given elliptic curve $E/\CC(\lambda)$ which are $K_\infty$-rational should be finitely generated.  To provide an interesting contrast, let $F/\CC(\lambda)$ be the field of Laurent series in $\lambda$, $F=\CC((\lambda))$.  Then an application of Tate's non-archimedean uniformization of elliptic curves shows that $E_0(F)$ is a group containing a cyclic subgroup of order $m$, for \emph{each} $m$.  That is, the group of germs of multiplier sections at $\lambda=0$ is far from finitely generated.

\begin{proof}[Proof of Proposition~\ref{X_1(2) bir}]
One way to construct an explicit affine curve birational to $X_0(N)$ is to consider the projection of the curve $Y_1(N)$ onto the $(a, b)$-plane.  This is given by the resultant of
\[\Phi_2=\frac{f(f(z))-z}{f(z)-z}\qquad\text{and}\qquad \frac{\partial f^2}{\partial z}-\lambda,\]
as polynomials in $z$.  This resultant is the square (since this projection is a double-cover) of the polynomial
\begin{multline*}
R=729+972a-432a^3-108a^4+48a^5+16a^6+1458b^2+1215b^2a+324b^2a^2\\+216b^2a^3+729b^4-243\lambda-216\lambda a+48\lambda a^3+12\lambda a^4\\-162\lambda b^2+81a\lambda b^2+27\lambda^2+12a\lambda^2-\lambda^3.
\end{multline*}
Let $C=E_0\setminus\{\Ocal, (0,0)\}$, where $\Ocal$ is the point at infinity, and let $C'\subseteq\AA^2$ be the locus of vanishing of $R$.  
One can check, with a computer algebra package such as Maple, that the functions $a$ and $b$ defined above actually provide a morphism from $C$ to $C'$ (that is, that the function $R(a, b)$ vanishes identically on $C$).  Now, note that there is precisely one point at infinity on the closure of $C'$ in $\PP^2$, and it is a nodal singularity.  The map $C\rightarrow C'$ extends to a morphism sending $\Ocal$ and $(0,0)$ to this nodal singularity.  Thus, the singular point on the projective closure of $C'$ corresponds to two points on the normalization of $C'$, each of which has precisely one preimage under the morphism induced by this rational map.  That is to say, the morphism $C\rightarrow C'$ induces an isomorphism between $E_0$ and the normalization of the projective completion of $C'$.

We now know that $X_0(2)$ is isomorphic to the elliptic curve $E_0$, and we turn our attention to $X_1(2)$.  Note that
\[\Phi_2(a, b, z)=a^2z^2+2z^4a+az^2+2azb+a+z^6+z^4+2z^3b+z^2+bz+b^2+1\]
and
\begin{multline*}
\frac{\partial f^2}{\partial z}-\lambda =9z^8+21z^6a+15z^4a^2+18z^5b+24z^3ba \\+3a^3z^2+6a^2zb+9b^2z^2+3b^2a+3az^2+a^2-\lambda.
\end{multline*}
Composing the maps $a, b\in \CC(E_0)$ defined above with the isogeny $E_1\rightarrow E_0$, defined by
\[
u=\frac{e^2}{4d^2} \quad\text{ and }\quad
v=\frac{e(d^2-4\lambda)}{8d^2},\]
we see that we have a simulateneous solution to the equations above with \[z=\frac{\sqrt{-2}(d^2-6d +8\lambda)}{6e}.\]
In other words, we have constructed a map $E_1\rightarrow X_1(2)$ which makes the diagram
\[\begin{CD}
E_1 @>>> E_0\\
@VVV @VVV\\
X_1(2) @>>> X_0(2)
\end{CD}\]
commute.  Since the rightmost map is an isomorphism, and the two horizontal maps have the same degree, the leftmost map also has degree 1, and is therefore an isomorphism.
\end{proof}

\begin{remark}
The map $X_1(N)\rightarrow X_0(N)$ sending a point of period $N$ to the cycle it generates is an obvious map from the point of view of moduli spaces.  The fact that, in the case $N=2$, these curves are both elliptic, however, means that there is a dual map $X_0(N)\rightarrow X_1(N)$, also unramified and of degree 2.  It would be interesting to understand the interpretation, if any, of this map in terms of the underlying dynamics.  That is, given two cubic polynomials with marked 2-cycles, how does the cubic polynomial with a marked point of period 2 arise?
\end{remark}

The Mordell-Weil theorem tells us that the rational points on $E_0$ or $E_1$ over any finite extension of $\CC(\lambda)$ has the structure of a finitely generated group.  We wish to compute this structure over extensions of the form $K_n=\CC(\lambda^{1/n})$, and indeed over $K_\infty=\bigcup_{n\geq 1}K_n$.  We will focus on the arithmetic of $E_0(K_\infty)$, given the obvious map of moduli spaces $X_1(N)\rightarrow X_0(N)$.

\begin{proposition}\label{prop:mw group}
We have $E_0(K_\infty)=E_0(K_{12})$.  Moreover, if $t^{12}=\lambda$ and $\zeta^4-\zeta^2+1=0$ (i.e., $\zeta$ is a primitive 12th root of unity), then $E_0(K_{12})$ is (abstractly) isomorphic to $\ZZ^3\oplus (\ZZ/2\ZZ)^2$, generated by the following points:
\begin{gather*}
P=\big(-1+(i-1)t^3+it^{6},(1-i)(t^3+i)(t^3+1)t^3\big)\\
R_1=(t^4-1, \zeta^9t^4(t^4-1))\\
R_2=(\zeta^4t^4-1, \zeta t^4(\zeta^4t^4-1))\\
T_1=\big(0, 0\big)\\
T_2=\big(-1+t^{6}, 0\big).
\end{gather*}
\end{proposition}

One can easily check that the last two points each have order 2, and it will be shown below that the first three are independent points of infinite order.  The first step in proving that these points in fact generate $E_0(K_\infty)$ is to
 prove that $E_0(K_\infty)$ has the claimed torsion subgroup.

\begin{lemma}
Let $E_0/\CC(\lambda)$ be the elliptic curve described above.  Then
\[E_0(K_\infty)_{\mathrm{Tors}}=E_0[2]\subseteq E_0(K_2).\]
\end{lemma}

\begin{proof}
Let $\XEll_1(N)$ denote the usual modular curve parametrizing elliptic curves with a point of order $N$.  If $E_0(K_\infty)$ contains a point of order $p$, where $p$ is an odd prime, then so does $E_0(K_n)$, for some $n$.  If we denote our chosen $n$th root of $\lambda$ by $\alpha$, the elliptic curve $E_0/K_n$ has $j$-invariant
\[j_{E_0}=\frac{64(3\alpha^n-4)^3}{\alpha^{n}(\alpha^n-1)^2}.\]
Since $E_0(\CC(\alpha))$ contains a point of order $N$, this $j$-map must factor as $j_{E_0}=j_p\circ\phi$, where $\phi:\PP^1\rightarrow \XEll_1(p)$, and $j_p:\XEll_1(p)\rightarrow\PP^1$ is the $j$-map associated to $\XEll_1(p)$.  By well-known facts about modular curves, $j_p$ has exactly $\frac{p-1}{2}$ simple poles, and $\frac{p-1}{2}$ poles of order $p$.  In particular, since $j_{E_0}$ has only one pole of order greater than 2, the factorization above is possible only when $\frac{p-1}{2}=1$, i.e., when $p=3$.  Suppose $p=3$.  In this case, $j_p$ has one simple pole, and one pole of order 3.  The degree of $j_{E_0}$, which is $3n$, must be divisible by the degree of $j_p$, which is 4, so $4\mid n$.  Also, the $n$ distinct poles of $j_{E_0}$ of order $2$ come from $n$ pts which each map to the cusp of $\XEll_1(p)$ at which $j_p$ has a simple pole, with multiplicity 2.  This means that the degree of $\phi$ is $2n$.  On the other hand, the one pole of $j_{E_0}$ of order $n$ maps to the other cusp of $\XEll_1(p)$ with multiplicity $n/3$.  This means the degree of $\phi$ is 1/3.  Impossible.

It remains to show that $E_0(K_\infty)_\mathrm{Tors}$ contains no point of order 4.  In this case $j_{E_0}$ factors through the map $j_4:\XEll_1(4)\rightarrow\PP^1$, which has a pole of order 4, and a simple pole.  Again, the pole of order $n$ of $j_{E_0}$ comes from a totally ramified point, with ramification index $n/4$, above one of the cusps of $\XEll_1(4)$.  The $n$ poles of order 2 each correspond to a point over the other cusp at which $\phi$ has ramification index $2$.  So $\phi$ must have degree $n/4$, on the one hand, and $2n$, on the other.

So we have shown that $E_0(K_\infty)\subseteq E_0[2]$.  The other inclusion is obvious from $E_0[2]\subseteq E_0(K_2)$.
\end{proof}

Next, we show that the rank of $E_0$ over the fields $K_n$ is no greater than expected.

\begin{lemma}\label{lem:fastenberg}
For any $n$, 
\[\rank\left(E/\CC(\lambda^{1/n})\right)\leq \begin{cases}0 &\text{if }\gcd(n, 6)=1\\ 1 &\text{if } \gcd(n, 6)=2\\ 2&\text{if }\gcd(n, 6)=3\\ 3&\text{if }\gcd(n, 6)=6.\end{cases}\]
\end{lemma}

\begin{proof}
We employ a result of Fastenberg \cite{fastenberg}, with some slight improvements.  Let $\pi:\Ecal\rightarrow\PP^1$ be a non-isotrivial elliptic surface.  Furthermore, let $\Ecal^t$ be the fibre of $\Ecal$ above $t\in\PP^1$, let $f_t$ be the local conductor, so that
\[f_t=\begin{cases}0 & \text{if }\Ecal\text{ has good reduction at }t\\
1 & \text{if }\Ecal\text{ has multiplicative reduction at }t\\
2 & \text{if }\Ecal\text{ has additive reduction at }t,
\end{cases}\]
 and let $e_t$ be the Euler characteristic of $\Ecal^t$.  For $t=0$ or $\infty$ let 
 \[n_t=\begin{cases}n &\text{ if }E^t\text{ has type }\I_n\text{ or }\I_n^*,\\
 0 & \text{otherwise,}\end{cases}\]  and set
\[\gamma=\sum_{t\neq 0, \infty}(f_t-e_t/6)-\frac{n_0+n_\infty}{6}.\]
Finally, let $\kappa(n)$ be the largest prime-power divisor of $n$.  Then Fastenberg's Theorem states that if $\gamma<1$, we have
\begin{equation}\label{eq:fastenberg}
\rank(E/\CC(t^{1/n}))\leq \sum_{\substack{d\mid n\\ \kappa(d)<\frac{2}{1-\gamma}}}\phi(d),
\end{equation}
where $\phi$ is the Euler totient function.

Note that, in the case of the elliptic curve $E_0/\CC(\lambda)$, there are precisely three singular fibres, above $t=0, 1, \infty$, and their reduction types are:
\[\begin{array}{c|c|c|c}
t & \text{type} & f_t & e_t\\\hline
0 & \I_1 & 1 & 1\\
1 & \I_2 & 1 & 2 \\ 
\infty & \III^* & 2 & 9
\end{array}\]
In particular, $n_0+n_\infty=1$, and so we have $\gamma=\frac{1}{2}<1$.  The sum in \eqref{eq:fastenberg} is over divisors $d\mid n$ with $\kappa(d)<4$, and the only integers with $\kappa(d)<4$ are $d=1, 2, 3$, and $6$.  The bound given by Fastenberg's theorem, then, is 
\[\rank\left(E/\CC(\lambda^{1/n})\right)\leq \begin{cases}\phi(1)=1 &\text{if }\gcd(n, 6)=1\\ \phi(1)+\phi(2)=2 &\text{if } \gcd(n, 6)=2\\ \phi(1)+\phi(3)=3&\text{if }\gcd(n, 6)=3\\ \phi(1)+\phi(2)+\phi(3)+\phi(6)=6&\text{if }\gcd(n, 6)=6.\end{cases}\]

To improve these bounds, we need to look more closely at the proof of  the theorem.  Let $\pi:\Ecal\rightarrow \PP^1$ be the N\'{e}ron model of $E_0$, and let $\pi_r:\Ecal_r\rightarrow \PP^1$ be the base extension by the map $z\mapsto z^r$.  The map on the base gives rise to an automorphism $\sigma_r:\Ecal_r\rightarrow\Ecal_r$.  The group of sections $\Ecal(\PP^1)$ on $\Ecal$ is isomorphic to the group $E_0(\CC(\lambda))$, while $\Ecal_r(\PP^1)\cong E_0(\CC(\lambda^{1/r}))$.  Now, as is well known (see \cite{fastenberg} for notation), there is an isomorphism of linear spaces
\[\Ecal(\PP^1)\otimes\QQ\cong H^1(\PP^1, R^1\pi_*\QQ)\cap H^{1, 1}(\Ecal, \CC),\]
and similarly for $\Ecal_r$.  Fastenberg establishes the bound above by studying the action of $\sigma_r$ on $H^1(\PP^1, R^1\pi_{r*}\CC)$.  Specifically, it is shown that
\[H^1(\PP^1, R^1\pi_{r*}\CC)=\bigoplus_{d\mid r}W_d^k,\]
where $k$ is the number of singular fibres on $E_0$ above $\PP^1\setminus\{0, \infty\}$ (in our case, $k=1$), and $W_d$ is the subspace generated by eigenvectors of $\sigma_r$ with eigenvalue a primitive $d$th root of unity.  The bound on the rank comes from restricting which of these eigenspaces may lie in $H^1(\PP^1, R^1\pi_{r*}\QQ)\cap H^{1, 1}(\Ecal_r, \CC)$.

For example, the $\phi(1)$ term in the bounds above comes from the eigenspace with eigenvalue 1.  If, however, this turns out to be a subspace of 
\[H^1(\PP^1, R^1\pi_{r*}\QQ)\cap H^{1,1}(\Ecal_r, \CC)\cong\Ecal_r(\PP^1)\otimes \QQ,\] then it is one fixed by $\sigma_r$, and so it corresponds to a one-dimensional subspace of $\Ecal(\PP^1)\otimes \QQ$.  In particular, then, we have $\rank(\Ecal(\PP^1))\geq 1$.  It turns out, however, that this is impossible.  The surface $\Ecal$ is rational, and so its N\'{e}ron-Severi group has rank 10.  If $m_t$ is number of components of the fibre $\Ecal^t$, it follows from Shioda's formula \cite[Corollary~5.3]{shioda} that
\[\rank(\Ecal(\PP^1))=\rank(\NS(\Ecal))-2-\sum_{t\in \PP^1}(m_t-1)=0\]
(recalling from above that $\Ecal$ has three bad fibres, of type $\I_1$, $\I_2$, and $\III^*$, respectively).
In particular, the fixed space of $\sigma_r$ in $H^1(\PP^1, R^1\pi_{r*}\CC)$ does \emph{not} lie in $H^1(\PP^1, R^1\pi_{r*}\QQ)\cap H^{1,1}(\Ecal_r, \CC)$, and so there is no contribution to the rank of $\Ecal_r(\PP^1)$ from this space; all of the rank bounds above may be reduced by 1. 

 This gives the bounds claimed above, except in the case where $6\mid n$, but this case follows by a similar argument.  It suffices to show that $\rank(\Ecal_6(\PP^1))\leq 3$, since any eigenspaces contributing to the rank of $\Ecal_{6n}(\PP^1)$ already contributes to the rank of $\Ecal_6(\PP^1)$.  But $\Ecal_6$ is an elliptic K3 surface (over a field of characteristic 0), and so its N\'{e}ron-Severi group has rank at most 20.  This surface has fibres of type $\I_2$ above all $t\in\PP^1$ with $t^6=1$, a fibre of type $\I_6$ above $t=0$, and a fibre of type $\I_0^*$ above $t=\infty$.  Thus,
\[\rank(\Ecal_6(\PP^1))\leq 20-2-6\cdot(2-1)-(6-1)-(5-1)=3,\]
proving the lemma.
\end{proof}

Our next lemma describes the group of $\CC(\lambda^{1/4})$-rational points on $E_0$, which we by now know to have rank at most 1.

\begin{lemma}
Let $t^4=\lambda$.  Then, over $\CC(t)$, the Mordell-Weil group of $E_0$ is exactly the group generated by the four points of order 2, along with
\[P=\big(-1+(i-1)t+it^{2},(1-i)(t+i)(t+1)t\big)\]
(note that $i=\zeta^9$).
\end{lemma}

\begin{proof}
First, we compute the pairing $\langle P,  P \rangle$, according to the method developed in \cite{shioda}.  If $\Ecal$ is the N\'{e}ron model of $E_0$ over $K_{4}$, we have
\[\langle P, P \rangle = 2\chi(\Ecal)+2(P\Ocal)-\sum \operatorname{contr}_v(P),\]
where $\chi(\Ecal)=-(\Ocal)^2$ is the arithmetic genus of $\Ecal$, $(P\Ocal)$ is the intersection pairing of the sections defined by $P$ and the identity $\Ocal$, and $\operatorname{contr}_v(P)$ is the contribution from the fibre above $v$.  The surface $\Ecal$ is rational, so $\chi(\Ecal)=1$.  Also, since the coordinates of $P$ are polynomials in $t$ of degree at most 2 and 3, respectively, $P$ misses the $\Ocal$-section everywhere; we have $(P\Ocal)=0$.  It suffices to compute the contributions at the places of bad reduction.  The N\'{e}ron model $\Ecal$ has type $\mathrm{I}_4$ reduction at $t=0$, and $P$ has order 4 in the component group, since the component group has order 4, and 
\[2P=(-1, t^2)\]
reduces to the singular point modulo the place defined by $t=0$.  Thus, $\operatorname{contr}_{(t)}(P)=1(4-1)/4$.  At the places defined by $t=-1$ and $t=-i$, the point $P$ reduces to the singular point.  At these places the component group has order 2, and so $P$ must be on the only non-trivial component, hence $\operatorname{contr}_v(P)=\frac{1}{2}$ at these places.  The only other fibres of bad reduction are at $t= 1, i$ ($\Ecal$ has good reduction at $t^{-1}=0$), and $P$ is on the non-singular component at these places.  We have
\[\langle P, P\rangle = 2\chi(\Ecal)+2(P\Ocal)-\sum \operatorname{contr}_v(P)=2-\frac{3}{4}-\frac{1}{2}-\frac{1}{2}=\frac{1}{4}.\]
Now, suppose that $P=mQ+T$, for $Q\in E(K_4)$ and $T\in E[2]$.  Then we have
\[\langle Q, Q\rangle \geq \frac{1}{4},\]
since the denominators arising in the local contributions all divide $4$.  On the other hand, the bilinearity of the pairing gives us
\[\frac{1}{4}=\langle P, P\rangle = \langle mQ+T, mQ+T\rangle= m^2\langle Q, Q\rangle\geq\frac{m^2}{4}.\]
  In particular, $m=\pm1$, and it follows that $P$ generates $E(K_4)$, modulo torsion.
\end{proof}

\begin{lemma}  Let $t^3=\lambda$.
The Mordell-Weil group of $E_0$ over $\CC(t)$ is generated by
\begin{gather*}
R_1=(9t-3, 27\zeta^9t(t-1))\\
R_2=(9\zeta^4t-3, 27\zeta t(\zeta^4 t-1))\\
T=(6, 0).
\end{gather*}
\end{lemma}

\begin{proof}
Lemma~\ref{lem:fastenberg} tells us that the rank at most 2, and  the height pairing matrix (which may be computed as in the previous lemma, or using MAGMA) is
\[\left(\begin{array}{cc}1/3& -1/6\\-1/6&1/3\end{array}\right).\]
The determinant of this is $\frac{1}{12}$.  On the other hand, it is clear from the possible contributions from bad places that we must have $6\langle Q_1, Q_2\rangle\in\ZZ$ for any points $Q_i\in E_0(\CC(t))$, and so  any lattice $L\subseteq E_0(\CC(t))$ must satisfy $\det(L)\geq\frac{1}{36}$.  If the lattice generated by $R_1$ and $R_2$ has index $m$ in some larger lattice, we must have $\frac{1}{12}\geq\frac{m^2}{36}$, and so $m^2\leq 3$.  This gives $m=1$, confirming that $R_1$ and $R_2$ span $E_0(\CC(t))$, modulo torsion.
\end{proof}

At this point, we have an explicit description of $E_0(\CC(\lambda^{1/4}))$ and $E_0(\CC(\lambda^{1/3}))$, along with rank bounds for $E_0(\CC(\lambda^{1/n}))$, which we know must be sharp.  After some small computation, we will be in a position to prove Proposition~\ref{prop:mw group}.

\begin{lemma}\label{lem:index}
Let $P$, $R_1$, and $R_2$ be as above, and let $H_1$ be the subgroup of $E(K_{12})$ generated by $P$ and $E[2]$, let $H_2$ be  the subgroup generated by $R_1$, $R_2$, and $E[2]$, and let $H_3$ be the subgroup generated by all of these points.  Then
\begin{enumerate}
\item the group $E(K_{12})/H_1$ contains no non-trivial elements of order 3;
\item the group $E(K_{12})/H_2$ contains no non-trivial elements of order 2 or 4;
\item the group $E(K_{24})/H_3$ contains no non-trivial elements of order 2.
\end{enumerate}
\end{lemma}

\begin{proof}
This is a fairly straightforward claim to verify computationally.  We will begin with the last claim.  Let $t^{24}=\lambda$, so that the last claim is that $E_0(\CC(t))$ contains no points of order 2 over $H_3$.  In other words, we wish to show that if $Q\in E(\CC(t))$ and
\[2Q=m P+n_1R_1+n_2R_2+T,\]
for $m, n_i$ some integers, and $T\in E[2]$, then $Q$ is already expressible in this form.  Note that we are free to translate $Q$ by elements of $H_3$ (since this does not change the image in the quotient group), so we may freely assume that $0\leq m, n_i\leq 1$.

Now, it is a well-known fact that the map $[2]:E\rightarrow E$ induces a map $\phi:\PP^1\rightarrow\PP^1$, by $x\circ [2]=\phi\circ x$.  Writing $\phi=F(z, t)/G(z, t)$, for polynomials $F, G\in\ZZ[z, t]$, any solution
\[2Q=m P+n_1R_1+n_2R_2+T\]
as above yields a solution to
\[F(z, t)-x(m P+n_1R_1+n_2R_2+T)G(z, t),\]
with $z\in \CC(t)$.  In other words, the polynomial above (in variables $z$ and $t$) has a linear factor over $\CC$.  It is a simple matter to write a MAGMA script (this script can be found in an appendix) which computes $x(mP+m_1R_1+m_2R_2+T)$, for each of the 32 possible choices, and checks to see if the resulting bivariate polynomial is reducible.  In fact, we see that the resulting polynomial is geometrically irreducible unless $m=n_1=n_2=0$ and $T=\Ocal$ (this is not a counterexample to our claim, since the solutions in this case are solutions to $2Q=\Ocal$, which are already contained in $H_3$).

Note that this computation also treats (2).  Claim (1) is treated by a similar computation, for which MAGMA code appears at the end of this paper.
\end{proof}

We are now in a position to prove the proposition describing the arithmetic of $E_0$ over $K_\infty$.  The height pairing shows that $R_1$ and $R_2$ are independent, while $P$ is independent from these points, since its field of definition intersects that of $R_1$ and $R_2$ only on $\CC(\lambda)$, and here the rank of $E$ is 0.  Thus we have shown that the curve has the expected rank over each field $\CC(\lambda^{1/n})$ (since $2P\in E_0(\CC(\lambda^{1/2}))$ is a point of infinite order).

First we will show that the points in question generate the Mordell-Weil group of $E_0$ over $K_{12}$.  First, note that we have a complete description of $E_0(K_4)$ and $E_0(K_3)$ by the lemmas above.

  Suppose that $Q\in E_0(K_{12})$ is not in the subgroup $H$ generated by these points.  Since the rank of $E_0(K_{12})$ is the same as that of $H$, it must be that $Q$ is torsion over $H$.  Let $M\geq 1$ be the least positive integer such that $MQ\in H$, say
\[MQ=nP+m_1R_1+m_2R_2+T.\]
Without loss of generality, we will suppose that 
\[0\leq n, m_1, m_2<M,\]
since $Q$ translated by an element of $H$ will still have order $M$ over $H$.
Then, if $\sigma$ generates the Galois group of $K_{12}/K_4$,
\[M\operatorname{Tr}_{K_{12}/K_{4}}(Q)=Q+Q^\sigma+Q^{\sigma^2}=3nP+T',\]
for some $T'\in E_0[2]$.  Since $E_0(K_4)$ is generated by $P$, modulo torsion, it follows that $M\mid 3n$, and so (since $0\leq n<M$), we have
\[n\in \left\{0, \frac{M}{3}, \frac{2M}{3}\right\}.\]
Similarly, by computing the trace of $MQ$ to $K_3$, we have
\[m_1, m_2\in \left\{0, \frac{M}{4}, \frac{M}{2}, \frac{3M}{4}\right\}.\]
Our equation, then, becomes
\[MQ=\frac{M\delta}{3}+\frac{M\epsilon_1}{4}R_1+\frac{M\epsilon_2}{4}R_2+T,\]
with $\delta=0$, 1, or 2, and $\epsilon_i=0$,1, 2, or 3.  In other words,
\[12Q=4\delta P+3\epsilon_1 R_1+3\epsilon_2 R_2+T',\]
for some $T'\in E_0[2]$.
Now, 
\[\delta P+T'=3(4Q-\delta P-\epsilon_1R_1-\epsilon_2R_2)=3Q',\]
for some $Q'\in E_0(K_{12})$.  By Lemma~\ref{lem:index}, this is possible only if $\delta=0$ (given that $0\leq \delta<3$).  Now the equation becomes
\[4Q=\epsilon_1 R_1+\epsilon_2 R_2+T'.\]

Now, to show that $E_0(K_{12})$ is, in fact, all of $E_0(K_\infty)$, suppose that \[Q\in E_0(K_\infty)\setminus E_0(K_{12}).\]  Then $Q\in E_0(K_{12m})$ for some $m$, and we'll let $m$ be the least such $m$.  Let $\sigma$ be the generator of the Galois group of $K_{12m}/K_{12}$.  Then, since $E_0(K_{12m})\supseteq E(K_{12})$, and both groups have the same rank, we have $MQ\in E_0(K_{12})$ for some $M\geq 2$.  It follows that $Q-Q^\sigma\in E_0[M]\cap E(K_\infty)$.  Since we know that this group is exactly $E_0[2]$, it must be the case that $Q^\sigma=Q+T$ for some $T\in E_0[2]$.  However, since $E_0[2]\subseteq E(K_{12})$, we have
\[Q^{\sigma^2}=(Q+T)^{\sigma}=Q^\sigma+T=Q.\]
In particular, $Q$ is quadratic over $K_{12}$.  But we have shown that there is no point $Q\in E_0(K_{24})$ with $2Q\in H=E_0(K_{12})$.  It follows that $E_0(K_\infty)=E_0(K_{12})$, as claimed.

\begin{remark}
We have focussed on the spaces $X_1(N)$ and $X_0(N)$, but again the moduli space $\Pcal_3(2)$ is fairly easy to describe from this.  The  curve $X_1(2)$ is an elliptic curve, and the action of $\PSL_2$ corresponds to the automorphism $[-1]:X_1(2)\rightarrow X_1(2)$, as can be seen form the explicit formulas for $a$, $b$, and $z$ above.  In particular, the quotient $P_1(2)$ is isomorphic to $\PP^1$ (by the morphism $x:X_1(2)\rightarrow P_1(2)$, in the usual Weierstrass coordinates), and similarly for $P_0(2)$.  Thus, the moduli space $\Pcal_3( 2)$ is a ruled surface over $\PP^1$, and hence is birational  to $\PP^1\times \PP^1$.  \end{remark}


\section{The case $N\geq 3$}

The purpose of this section is to prove that $X_1(N)$ has genus at least 2 when $N\geq 3$, from which Theorem~\ref{th:N>3} will follow.  The following proposition establishes this for $N\geq 5$ and $N=3$.  Proposition~\ref{prop:genera when N is even} derives an even stronger bound in the case where $N$ is even, treating in particular the case $N=4$.  This proposition also gives lower bounds on the genera of $X_0(N)$ and $P_1(N)$ (in the case where $N$ is even).  All of these results are obtained using the Riemann-Hurwitz formula, and studying the quotients of these curves by various automorphisms.  Without an understanding of the points ``at infinity'' on $X_1(N)$, it seems impossible to do better than this.

\begin{proposition}\label{prop:g(X_1(N)) grows}
The genus of $X_1(N)$ is at least $\frac{N}{2}-1$, while the genus of $X_1(3)$ is at least 5.
\end{proposition}

\begin{proof}
Although computing the genus of $X_1(3)$ directly already poses a significant computational challenge, we may simplify this by specializing $\lambda$.  In particular, the curve obtained by specializing $X_0(3)$ at $\lambda=1$ is reducible (for obvious reasons), with a component of genus 5.  It follows that the genus of $X_0(3)$, and hence of $X_1(3)$, is at least 5.

Now, consider again the variety $X_1'(N)$, and in particular the (singular) projective model given by \eqref{eq:X' singular}.  There are (at least) two automorphisms acting on $X_1'(N)$, namely those induced by the action on the singular model by
\[\sigma:[u, v, s, z_1, ..., z_N]\mapsto [\zeta^3 u, \zeta^2 v, s, z_1, ..., z_N]=[\zeta u, v, \zeta^{-2} s, \zeta^{-2}z_1, ..., \zeta^{-2}z_N],\]
for $\zeta$ some fixed primitive sixth root of unity, 
and
\[\tau:[u, v, s, z_1, ..., z_N]\mapsto [u, -v, s, -z_1, ..., -z_N]=[-u, v, -s, z_1, ..., z_N].\]
The group $\langle \sigma, \tau\rangle\subseteq\operatorname{Aut}(X'_1(N))$ act freely in general, and the other varieties are quotients of $X'_1(N)$ by these various groups:
\[\begin{CD}
X'_1(N) @>/\langle\sigma\rangle>> X_1(N)\\
@V/\langle\tau\rangle VV @VV/\langle  \tau\rangle V \\
P'_1(N) @>>/\langle \sigma\rangle > P_1(N)
\end{CD}\]
(we abuse notation, and let $\sigma$ and $\tau$ also denote the automorphisms induced on $P'_1(N)$ and $X_1(N)$ by the corresponding maps on $X_1(N)$).
The points on the singular model of $X'_1(N)$ with $s=0$, however, are fixed by $\sigma^3\tau$.  It follows that non-singular points among these correspond to places on $X'_1(N)$ where the map \[X'_1(N)\rightarrow P_1(N)= X'_1(N)/\langle\sigma, \tau\rangle\] ramifies with index $e=2$ (\emph{a priori}, the singular points at which $\sigma^3\tau$ acts trivially might blow up into pairs of points on the normalization which are swapped by $\sigma^3\tau$).  Note that the points
\[[u, v, s, z_1, ..., z_N]=[\xi, 1, 0, \xi^4, -2\xi^4, ..., -2\xi^4],\]
for $\xi^6=1$, and their images under iteration by $\hat{f}$ are non-singular, which follows from examining the Jacobian matrix, exactly as in the proof of Proposition~\ref{prop:X_1(N) irred}.
Both $\sigma$ and $\tau$ act freely on these points, and they are mapped to $N$ (resp.\ $3N$) non-singular points on $X_1(N)$ (resp.\ $P'_1(N)$).  Here, however, $\sigma^3\tau$ acts trivially on them, and so we have $N$ points at which the map $X_1(N)\rightarrow P_1(N)$ ramifies (with $e=2$), and $3N$ points at which $P'_1(N)\rightarrow P_1(N)$ ramifies.  This gives the estimates (by Riemann-Hurwitz)
\[2g(X_1(N))-2\geq -4+N\]
and
\[2g(P'_1(N))-2\geq -12+3N.\]
\end{proof}

The previous proposition tells us that the genera of the curves $X_1(N)$ grow at least linearly in $N$.  This is enough for the proof of Theorem~\ref{th:N>3}, but it seems likely that the genera grow much more rapidly.  If $N$ is even, we can improve significantly on Proposition~\ref{prop:g(X_1(N)) grows}, as well as give lower bounds on the genera of $X_0(N)$ and $P_1(N)$.

\begin{proposition}\label{prop:genera when N is even}
Let $N=2n$, let $\theta$ be the completely multiplicative function defined by $\theta(2)=0$, and $\theta(p)=-p$ for any 
odd prime $p$, and let
\[\omega(n)=\sum_{d\mid n}\theta\left(\frac{n}{d}\right)2d3^d.\]
  Then the genera of $X_1(N)$, $X_0(N)$, and $P_1(N)$ satisfy
  \begin{gather*}
  g(X_1(N))\geq \frac{\omega(n)}{2}+1-2n\\
  g(X_0(N))\geq \frac{\omega(n)}{4n}-1\\
  g(P_1(N))\geq \frac{\omega(n)}{4}+1-2n.
  \end{gather*}
\end{proposition}

Before proceeding with the proof, it should be pointed out that the lower bound can be simplified, although also weakened, by the estimate $\omega(n)\geq n3^n$, which holds for all $n\geq 1$.  As a consequence, the genera in all three families grow at least exponentially in $N$, for $N$ even.

\begin{proof}[Proof of Proposition~\ref{prop:genera when N is even}]
Here we study ramification of certain maps
above $b=0$.  There are two automorphisms acting on the usual affine part of $X_1(N)$, namely
\[\tau:(a, b, z_1, ..., z_N)\mapsto (a, -b, -z_1, ..., -z_N)\]
and
\[\hat{f}:(a, b, z_1, ..., z_N)\mapsto (a, b, z_2, ..., z_1, z_N),\]
giving the diagram
\[\begin{CD}
X_1(N) @>/\hat{f}>> X_0(N)\\
@V/\tau VV @VV/\tau V\\
P_1(N) @>>/\hat{f}> P_0(N)
\end{CD}\]
(here we abuse notation somewhat again, and let $\tau$ stand for both the map on $X_1(N)$, and the induced map on $X_0(N)$, and similarly for $\hat{f}$).
The top map is unramified on the affine part (at least), while the left map is unramified except when $N=1$, at $(a, b, z)=(\lambda, 0, 0)$, which we may ignore, since we are taking $N$ even.  The right-hand map, however, is ramified exactly where (on the affine part) the polynomial $f_{a, b}$, with cycle $z_1, ..., z_N$, coincides with the polynomial $f_{a, -b}$, with cycle $-z_1, ..., -z_N$.  In other words, where $b=0$, and $-z_1$ is in the orbit of $z_1$.
 Note that if $b=0$, then $f_{a, 0}(-z)=-f_{a, 0}(z)$ for all $z$, and so $-z$ is in the forward orbit of the $N$-period point $z$ if and only if $N=2n$ is even, and $f^n_{a, 0}(z)=-z$.  So let $\omega(n)$ be the number of points \[(a, 0, z_1, ..., z_N)\in X_1(2n)\] such that $z_n=-z_1$ (we will prove shortly that $\omega(n)$ is the function defined in the statement of the proposition).  Then $X_0(N)$ contains $\omega(n)/N$ images of these points, and at each the right-hand map in the diagram ramifies with index $e=2$.  This gives a bound of
\[2g(X_0(N))-2\geq -4+\omega(n)/N.\]
This immediately gives the bound
\[2g(X_1(N))-2\geq N(2g(X_0(N))-2)\geq \omega(n)-4N.\]
Similarly, $P_1(N)$ has $\omega(n)/2$ points at which the map $P_1(N)\rightarrow P_0(N)$ ramifies with index $e=2$, giving the bound
\[2g(P_1(N))-2\geq -2N+\omega(n)/2.\] 
All that remains is to determine $\omega(n)$, that is, to show that it is the function defined in the theorem.

Let $V\subseteq \PP^{n+1}_{\CC[\lambda]}$ be the $0$-dimensional variety defined by the system of equations
\begin{gather*}
z_1^3-w^2z_1-s^2z_2=0\\
z_2^3-w^2z_2-s^2z_3=0\\
\vdots\\
z_n^2-w^2z_n+s^2z_1=0\\
(3z_1^2-w^2)^2\cdots(3z_n^2-w^2)^2-s^{4n}\lambda=0.
\end{gather*}
This variety parametrizes $2n$-cycles for the function $f(z)=z^3-w^2z$ with multiplier $\lambda$, and satisfying $f^n(z_1)=-z_1$.  By B\'{e}zout's Theorem, the number of points on $V$, counted with multiplicity, is precisely the product of the degrees of the defining equations, or in this case $4n3^n$.  Now, if $[w, s, z_1, ..., z_n]\in\PP^{n+1}$ were a solution to the above system with $s=0$, we would have $z_i\in \{0, \pm w\}$ for all $i$, by the first $n$ equations, and $z_i=\pm 3^{-1/2}w$ for \emph{some} $i$, by the last equation.  These conditions are incompatible, and so all of the points on $V$ lie within the affine open, call it $U\subseteq V$, defined by $s\neq 0$.  We dehomogenize with $s=1$, and choose an $\alpha\in\overline{\CC(\lambda)}$ with $\alpha^2=\lambda$.  Then $U$ is made up of two components, defined by the first $n$ equations above, along with either of
\[(3z_1^2-w^2)\cdots (3z_n^2-w^2)\pm \alpha=0.\]
These components are clearly disjoint, and by Lemma~\ref{lem:nonsingular} (applied with $\mu=-1$, and $\nu=\pm 1$), each is nonsingular.  Thus, the original variety $V$ is non-singular, and so contains precisely $4n3^n$ distinct points.

Now, let $V_n(\lambda)$ be the affine variety defined by
\begin{gather*}
z_1^3+az_1-z_2=0\\
z_2^3+az_2-z_3=0\\
\vdots\\
z_n^2+az_n+z_1=0\\
(3z_1^2+a)^2\cdots(3z_n^2+a)^2-\lambda=0.
\end{gather*}
The map $V\mapsto V_n(\lambda)$ can ramify at a point only if $w=0$ there.  If $w=0$, though, $f(z)=z^3$ and so periodic points for $f$ are roots of unity.  But then we have \[(3z_1^2-w^2)^2\cdots(3z_n^2-w^2)^2\neq \lambda,\]
since the left-hand-side is constant (with respect to $\lambda$).
It follows that that $V_n(\lambda)$ has $2n3^n$ points defined over $\overline{\CC(\lambda)}$.  However, some of these points satisfy $f^m(z_1)=z_1$ for $m<2n$.  If this is the case, then we must have $f^d(z_1)=-z_1$ for some $d\mid n$, and $n/d$ must be odd.  On the other hand, for each $d\mid n$ with $n/d$ odd, and each $d$th root $\gamma$ of $\lambda$, there is an embedding of $V_d(\gamma)$ into $V_n(\lambda)$ simply by $(z_1, ..., z_d, w)\mapsto (z_1, ..., z_d, z_1, ..., z_d, ..., w)$.  In other words, if $\omega(n)$ is the number of points in $V_n(\lambda)$ which correspond to actual $n$-cycles, we have
\[\sum_{d\mid n}\beta\left(\frac{n}{d}\right)\omega (d)=2n3^n,\]
where
\[
\beta(m)=\begin{cases}0 & 2\mid m\\
m & \text{otherwise}.\end{cases}
\]
The function $\theta$ defined in the statement of the proposition is the Dirichlet inverse of $\beta$,  and so we have
\[\omega(n)=\sum_{d\mid n}\theta\left(\frac{n}{d}\right)2d3^d,\]
as claimed.
\end{proof}

Note that this approach might be used to show that the \emph{surface} $\Pcal_3(N)$ is of general type, for $N$ even and large enough and even, but it will not do the same for $\Pcal_d(N)$ with $d\geq 4$.  The reason for this is that $\PSL_2$ acts non-freely on the $N$-periodic point $z$ for $f(z)=z^d+a_{d-2}z^{d-2}+\cdots+a_1z+a_0$ only if, for $\zeta$ a $(d-1)$th root of unity, $f$ is fixed under conjugation by $z\mapsto \zeta z$, and $f^{N/(d-1)}(z)=\zeta z$.  For $d\neq 1$, though, the subvariety on which this happens has codimension at least 2.

To summarize, we have the following lower bounds on the genera of various curves, where the first two columns are known:
\begin{center}
\begin{tabular}{|c|c|c|c|c|c|c|c|c|}\hline
$N$ & 1 & 2 & 3 & 4 & 5 & 6 & 7 &8\\\hline
$X_1(N)$ & \textbf{0} &\textbf{1} & 5 & 11 & 2 & 61 & 3 & 309 \\
$X_0(N)$ & \textbf{0} & \textbf{1} &  & 4 &  & 11 &  & 40\\
$P_1(N)$ & \textbf{0} & \textbf{0} &  & 6 &  & 31 &  & 155\\\hline
\end{tabular}
\end{center}
It seems unlikely that the genera of $X_1(N)$, for odd $N$, actually grow more slowly than for $N$ even, and we suspect that the actual genera should be of order $N3^N$.

\begin{remark}
Recent work of Bonifant, Kiwi, and Milnor \cite{bkm} has examined the curve parametrizing cubic polynomials with a marked critical point of specified period.  This curve is birational to the fibre above $\lambda=0$, on the appropriate curve $X_0(N)$.  It should be noted that, while the authors show that the Euler characteristic of this curve increases exponentially, this does not imply that the genus of $X_0(N)$ does.  Specifically, it is not known if these curves are irreducible, and even fibrations of generic genus one may have (reducible) fibres of arbitrarily large Euler characteristic.
\end{remark}

\section{Periodic points for polynomials of higher degree}

We now turn our attention to the proof that the moduli space $\Pcal_d(N)$ of polynomials of degree $d$, with a marked point of period $N$, is rational when $N\leq d+1$.   Slightly more generally, if $N_1, ..., N_s\geq 1$ are integers with $N_1+\cdots+N_s\leq d+1$, then the fibre product
\[\Pcal_d(N_1, ..., N_s)=\Pcal_d(N_1)\times_{\Pcal_d}\Pcal_d(N_2)\times_{\Pcal_d}\cdots\times_{\Pcal_d}\Pcal_d(N_s)\]
is birational to $\PP^{d-1}$.

\begin{proposition}
Let $d\geq 2$, and let $N_1, ..., N_s$ be non-negative integers with \[N_1+\cdots+N_s\leq d+1.\]  Then $\Pcal_d(N_1, ..., N_s)$ is rational.
\end{proposition}

\begin{proof}
For any $\overline{a}=(a_0, a_1, ..., a_d)\in\AA^{d+1}$, let
\[f_{\overline{a}}(z)=\sum_{i=0}^d a_iz^i.\]
The affine transformation $\phi(z)=\alpha z+\beta$, $\alpha\neq 0$, acts of $\overline{a}$ by sending it to $\overline{a}^\phi$, where
\[f_{\overline{a}^\phi}=\phi\circ f_{\overline{a}}\circ \phi^{-1}.\]
The moduli space $\Pcal_d(N)$ (up to birational equivalence) is simply the space of $(\overline{a}, z_1, ..., z_N)$ satisfying
\begin{eqnarray*}
f_{\overline{a}}(z_1)-z_2&=&0\\
f_{\overline{a}}(z_2)-z_3&=&0\\
&\vdots &\\
f_{\overline{a}}(z_N)-z_1&=&0,
\end{eqnarray*}
minus the hyperplanes $z_i=z_j$ for $i\neq j$,
modulo the action of the affine transformations \[(\overline{a}, z_1, ..., z_N)\mapsto (\overline{a}^\phi, \phi(z_1), ..., \phi(z_N)),\] which one may verify preserves the equations above.

Let $\sigma$ be the permutation of $\{1,2, 3, ..., N_1+N_2+\cdots+N_s\}$ which induces a cycle on $\{1, 2, 3, ..., N_1\}$, another on $\{N_1+1, N_1+2, ..., N_2\}$, \emph{et cetera},
so that $\sigma$ permutes the numbers $1, 2, ..., N_1+N_2+\cdots +N_s$ as $s$ disjoint cycles, of period $N_1$, $N_2$, etc., respectively.  If $M=N_1+N_2+\cdots N_s$, then the space $\Pcal_d(N_1, ..., N_s)$ is simply the quotient of the variety \[V\subseteq\Spec(\ZZ[z_1, ..., z_M, a_0, ..., a_d])\] defined by the equations
\begin{eqnarray*}
a_0+a_1z_1+a_2z_1^2+\cdots a_dz_1^d&=&\sigma(z_1)\\
a_0+a_1z_2+a_2z_2^2+\cdots a_dz_2^d&=&\sigma(z_2)\\
& \vdots &\\
a_0+a_1z_M+a_2z_M^2+\cdots a_dz_M^d&=&\sigma(z_M)
\end{eqnarray*}
by action defined above.
Permuting the equations, we may describe $V$ as the locus of the system defined by
\[
\left(\begin{array}{cccccc}
1 & z_{\sigma^{-1}(1)} & z_{\sigma^{-1}(1)}^2 & \cdots & \cdots & z_{\sigma^{-1}(1)}^d\\
1 & z_{\sigma^{-1}(2)} & z_{\sigma^{-1}(2)}^2 & \cdots & \cdots & z_{\sigma^{-1}(2)}^d\\
\vdots & \vdots & \vdots & && \vdots \\
1 & z_{\sigma^{-1}(M)} & z_{\sigma^{-1}(M)}^2 & \cdots &\cdots & z_{\sigma^{-1}(M)}^d\\
\vdots& \vdots & \vdots && & \vdots \\
0 &  0 & 0 & 1& 0 & 0\\
0 &  0 & 0 & 0& 1 & 0\\
0 & 0  & 0 & 0& 0 & 1
\end{array}\right)
\left(\begin{array}{c}
a_0 \\ a_1 \\ \vdots\\ a_{M-1}\\  a_M\\ \vdots \\ a_d
\end{array}\right)=
\left(\begin{array}{c}
z_1 \\ z_2 \\ \vdots \\ z_M\\ a_{M}\\ \vdots \\ a_{d}
\end{array}\right).
\]

The matrix on the left is a variant of the Vandermonde matrix, and one may check that it is invertible if and only if $z_i\neq z_j$ for all $i\neq j$.  Let $V'$ be the open subset of $\Spec(\ZZ[z_1, ..., z_M, a_0, ..., a_d])$ defined by deleting the hyperplanes $z_i-z_j$ for $i\neq j$, along with $a_d=0$, and let $U$ be its projection onto $\Spec(\ZZ[z_1, z_2, ..., z_M, a_M, ..., a_d])$.  Then the above matrix, which is invertible on $U$, gives an isomorphism between $U$ and $V'$.

For now, suppose that $M\geq 2$.  Then we claim that every point in $U$ is $\PSL_2$-equivalent to a unique point with $z_1=0$ and $z_2=1$, where the $\PSL_2$ action is that inherited from $V$.  To see that this is true, note that the full action of $\PSL_2$ on $V$ is given by the affine transformations $z\mapsto \alpha z+\beta$.  For any point $(z_1, z_2, ...)\in U$ we may conjugate by the map
\[z\mapsto \frac{1}{(z_1-z_2)}z-\frac{z_1}{(z_1-z_2)}\]
(which is defined since $z_1\neq z_2$) in order to translate the point to one of the form $(0, 1, ... )$.  If, on the other hand, two points of the form $(0, 1, ...)$ are conjugate by the map $z\mapsto \alpha z+\beta$, then $\alpha\cdot 0+\beta=0$, whence $\beta=0$, and $\alpha\cdot 1+\beta= 1$, implying $\alpha=1$.  We now know that each orbit in $U$ under the action of the affine transformations contains precisely one point of the form $(0, 1, ...)$, and that any point $(0, 1, ...)\in\AA^{d+1}$ appears.  In particular, this quotient of $U$ is isomorphic to $\AA^{d-1}$.  Since the moduli space $\Pcal_d(N_1, ..., N_s)$ is birational to the quotient of $V$ by the affine transformations, which in turn is birational to the quotient of $U$ by these transformations, which in turn is birational to $\AA^{d-1}$, we see that $\Pcal_d(N_1, ..., N_s)$ is a rational variety.

If $M=1$, then the variety is $\Pcal_d(1)$, the moduli space of polynomials with a marked fixed point, the quotient of $\AA^{d+2}=\Spec(\ZZ[a_d, ..., a_0, z])$ modulo the action of $\PSL_2$.
Note, as above, that every $\PSL_2$-equivalence class contains a point with $z=0$, which necessarily implies $a_0=0$.  Furthermore, if conjugation by $\phi(x)=\alpha x+\beta$ moves $(0, a_d, ..., a_1, 0)$ to $(0, \tilde{a_d}, ...,\tilde{a_1}, 0)$, then $\beta=0$, and $\tilde{a_i}=\alpha^{i-1}a_i$.  Restricting to the affine open defined by $a_2\neq 0$, then, every $\PSL_2$-equivalence class contains a unique point with $z=0$, $a_2=1$.  On the other hand, any choice of $a_d, a_{d-1}, ..., a_3, a_1$, with $a_2=1$, $a_0=z=0$, defines a polynomial of this form.  This gives an explicit birational map between $\Pcal_d(1)$ and $\AA^{d-1}$.

Finally, if $M=0$, the variety in question is simply $\Pcal_d$, the moduli space of polynomials of degree $d$.  This is clearly seen to be rational.  Let $U$ be the affine open subset of the moduli space consisting of $\PGL_2$-equivalence classes of polynomials whose barycenter, that is the average of roots with multiplicity, is not a fixed point.  Moving the barycenter to 0, and the value of 0 to 1 gives a polynomial of the form
\[a_dz^d+a_{d-2}z^{d-2}+\cdots+a_1z+1.\]
It is easy to check that two polynomials of this form are $\PSL_2$-conjugate if and only if they are actually equal, again giving a birational equivalence between this variety and $\AA^{d-1}$.
\end{proof}

The final case, the moduli space of polynomials of degree $d$, provides an interesting normal form for polynomials.  Unless the barycenter of a polynomial is a fixed point, the polynomial is $\PSL_2$-conjugate to a unique polynomial of the form
\[a_dz^d+a_{d-2}z^{d-2}+\cdots+a_1z+1,\]
providing an obvious isomorphism between this affine open and $\AA^{d-1}\setminus\{a_d=0\}$.  It also follows at once that, if the ground field is $F$, then the field of definition of the polynomial above is exactly $F(a_d, ..., a_1)$.  This is in contrast to the normal form for cubic polynomials used above, $f(z)=z^3+az+b$, where the field of definition is $F(a, b^2)$.  Note that conjugation by the M\"obius transformation $\psi(z)=bz$ translates $f$ to the polynomial $b^2z+az+1$.  More generally, the field of definition/moduli of the usual normal form
\[z^d+a_{d-2}z^{d-2}+\cdots+a_1z+a_0\]
is precisely $F(a_1, a_0a_2, a_0^2a_3, \ldots, a_0^{d-3}a_{d-2}, a_0^{d-1})$.  The disadvantage of this normal form, of course, is that it misses polynomials with a fixed barycenter, as well as providing an isomorphism with an open subset of $\AA^{d-1}$ which is not isomorphic to affine space in a natural way.

\end{document}